\documentclass[a4paper,10pt]{amsart}

\usepackage[english]{babel}
\usepackage[utf8]{inputenc}
\usepackage[T1]{fontenc}
\usepackage{lmodern}

\usepackage{amssymb}
\usepackage{amsmath}
\usepackage{amsthm}
\usepackage{bbm}

\usepackage[shortlabels]{enumitem}
\usepackage{tikz}
\usepackage{marginnote}
\usepackage{mathtools}
\usepackage{orcidlink}

\usepackage{hyperref}

\newcommand{\bbC}{\mathbb{C}}
\newcommand{\bbD}{\mathbb{D}}

\newcommand{\bbN}{\mathbb{N}}

\newcommand{\bbT}{\mathbb{T}}

\newcommand{\bbZ}{\mathbb{Z}}

\newcommand{\calA}{\mathcal{A}}

\newcommand{\calL}{\mathcal{L}}

\newcommand{\calU}{\mathcal{U}}

\newcommand{\N}{\bbN}

\newcommand{\C}{\bbC}

\newcommand{\Z}{\bbZ}

\DeclarePairedDelimiter{\norm}{\lVert}{\rVert}
\DeclarePairedDelimiter{\abs}{\lvert}{\rvert}

\DeclarePairedDelimiter{\set}{\lbrace}{\rbrace}

\DeclareMathOperator{\one}{\mathbbm{1}}

\DeclareMathOperator{\re}{Re}

\renewcommand{\Re}{\re}

\newcommand{\ui}{\mathrm{i}}

\theoremstyle{definition}
\newtheorem{definition}{Definition}[section]
\newtheorem{remark}[definition]{Remark}

\newtheorem{example}[definition]{Example}

\newtheorem{open_problem}[definition]{Open Problem}

\theoremstyle{plain}
\newtheorem{proposition}[definition]{Proposition}
\newtheorem{lemma}[definition]{Lemma}
\newtheorem{theorem}[definition]{Theorem}
\newtheorem{corollary}[definition]{Corollary}

\numberwithin{equation}{section}

\usepackage{etoolbox}

\makeatother

\begin{document}

\title[Cyclicity Conjecture]
{On the Cyclicity Conjecture}

\author{Alexander Dobrick \orcidlink{0000-0002-3308-3581}}
\address[A.~Dobrick]{Alexander Dobrick, Christian-Albrechts-Universit\"at~zu~Kiel (Alumni), Arbeitsbereich~Analysis, 24118 Kiel, Germany}
\email{alexander.dobrick.math@web.de}

\date{7 September 2026}

\begin{abstract}
    We prove that the peripheral spectrum of every positive operator on a complex Banach lattice is cyclic, thereby resolving the cyclicity conjecture without additional growth assumptions. Our proof is based on a spectral domination theorem derived from Vesentini's theorem on the subharmonicity of the spectral radius: when two positive operators have the same spectral radius and one dominates the other, the peripheral spectrum of the smaller operator is contained in that of the larger one. This extends an earlier result of Räbiger and Wolff. After passing to a suitable lattice extension, an averaging procedure based on torsion operators produces a decreasing sequence of positive minorants which retain the spectral radius and become asymptotically rotationally self-similar. An ultrapower argument turns this asymptotic relation into exact rotational self-similarity, allowing us to deduce cyclicity from spectral domination. The same approach also yields an alternative proof of Lotz's cyclicity theorem for Abel solvable operators.
\end{abstract}

\maketitle

\section{Introduction} \label{section:introduction}

Although the cyclicity problem has attracted sustained interest since the 1960s, this central conjecture in the spectral theory of positive operators, known as Perron--Frobenius theory, has remained open to this very day. The problem is the following:

\begin{open_problem}[Cyclicity Conjecture]
    Let $T$ be a positive operator on a complex Banach lattice $E$ with $\mathrm{r}(T)=1$. If $\lambda \in \sigma_{\mathrm{per}}(T)$, does it follow that $\lambda^k \in \sigma(T)$ for each $k \in \Z$?
\end{open_problem}

In finite dimensions, this observation goes back to the very beginning of classical Perron--Frobenius theory at the beginning of the twentieth century: the peripheral spectrum of every positive matrix is cyclic (see \cite{Frobenius1912, Perron1907}). When one passes to positive operators on infinite-dimensional Banach lattices, however, the situation becomes considerably more subtle. A systematic development of this theory took place in T\"ubingen around H.~H.~Schaefer during the 1960s and early 1970s; see in particular \cite{Schaefer1974}. During this period, substantial progress on the cyclicity problem was made by Lotz, Krieger, Scheffold and others (cf.~\cite{Krieger1969,Lotz1968,Scheffold1971}). Various further sufficient conditions were obtained subsequently (cf.~\cite{Glueck2016, Glueck2018, Grobler1995, Zhang1991}) for a remarkably broad collection of positive operators. One of the principal classical results is due independently to Lotz and Krieger (see \cite{Krieger1969, Lotz1968}): the peripheral spectrum is cyclic if the positive operator is Abel bounded. Lotz subsequently obtained cyclicity for the considerably larger class of Abel solvable operators. This class contains, in particular, the familiar case where the spectral radius is a pole of the resolvent and hence also compact positive operators with positive spectral radius. For a more in-depth historical account, we refer to \cite{Glueck2018} or to Bihlmaier's bachelor's thesis \cite{Bihlmaier2022}.

However, all these results seem to share a common theme: they require additional control of the asymptotic behaviour of the operator or of its resolvent, which is particularly striking since the conjecture itself contains no such assumption. There is also a basic conceptual difficulty behind this discrepancy. A peripheral spectral value of a general positive operator need not be an eigenvalue, so there is initially no vector on which the order structure can be exploited directly. Even if a peripheral eigenvector is available, positivity acts naturally only on its modulus. Thus, the problem requires a mechanism which converts order information into information about rotations of the complex spectrum without imposing additional growth assumptions.

In this article, we provide such a mechanism and obtain a positive answer to the cyclicity problem in full generality.

\begin{theorem} \label{theorem:cyclicity}
    Let $T$ be a positive operator on a complex Banach lattice $E$ such that $\mathrm{r}(T)=1$. Then $\lambda^k \in \sigma(T)$ for all $\lambda \in \sigma_{\mathrm{per}}(T)$ and $k \in \Z$.
\end{theorem}

Our proof is based on positive domination. The first main result of the article shows that peripheral spectral values pass from a positive operator $S$ to a positive operator $T$ under domination. In particular, if $0 \leq S \leq T$ and $\mathrm{r}(S)=\mathrm{r}(T)$, then
\begin{align*}
    \sigma(S)\cap \mathrm{r}(S)\bbT \subseteq \sigma(T)\cap \mathrm{r}(T)\bbT.
\end{align*}
The proof of this result is surprisingly short and relies on a classical theorem from abstract spectral theory. The inheritance of spectral properties under positive domination was studied systematically by R\"abiger and Wolff in \cite{Rabiger1997} and, as a side effect, we revisit one of their main results.

The second part of our argument exploits this first observation by constructing a sequence of carefully chosen minorants which become increasingly rotationally self-similar. Finally, an ultrapower replaces approximate rotational self-similarity by an exact relation that allows us to apply our spectral domination theorem.

\subsection*{Notation}
We write $\bbD = \{z \in \C : \abs z < 1\}$ and $\bbT = \{z \in \C : \abs z = 1\}$. For a complex Banach space $E$, we denote by $\calL(E)$ the Banach algebra of bounded linear operators on $E$ and by $I$ the identity operator. For $T \in \calL(E)$, we write $\sigma(T)$ for its spectrum, $\mathrm{r}(T)$ for its spectral radius, $\sigma_p(T)$ for its point spectrum and $\sigma_{\mathrm{app}}(T)$ for its approximate point spectrum. Its peripheral spectrum is $\sigma_{\mathrm{per}}(T) = \sigma(T) \cap \mathrm{r}(T)\bbT$, and $R(\lambda,T) = (\lambda I-T)^{-1}$ denotes its resolvent for $\lambda \notin \sigma(T)$.

For a complex Banach lattice $E$, we denote its positive cone by $E_+$ and the modulus of $x \in E$ by $\abs x$. The relation $0 \leq S \leq T$ means that both $S$ and $T-S$ are positive operators. We write $x \perp y$ if $\abs x \wedge \abs y = 0$.

\subsection*{Organisation of the article} In Section~\ref{section:peripheral-spectral-inheritance-under-positive-domination}, we prove our main spectral domination result. In Section~\ref{section:critical-minorants-and-cyclicity}, we prove the main result of the article by making use of a domination procedure. In Appendix~\ref{section:torsion-operators}, we give a brief exposition of torsion operators that play a prominent role in the proof of the main theorem. In Appendix~\ref{section:g-solvable-case}, we give an alternative proof of Lotz' weaker theorem that makes use of our technique.

\subsection*{AI disclosure statement}

During the preparation of this work, the author used an OpenCode harness, custom-written auxiliary skills and large language models from OpenAI's GPT-5.6 family for editorial assistance and to simplify some arguments. The statements of Proposition~\ref{proposition:vesentini} and Proposition~\ref{proposition:peripheral-rigidity} were found by an AI-assisted literature review. In an earlier version of the paper, the notion of rotational self-similarity had a different AI-inspired name. The author reviewed and verified all AI-assisted material and takes full responsibility for the content of the manuscript. All new contributions are due to the author.

\section{Peripheral spectral inheritance under positive domination} \label{section:peripheral-spectral-inheritance-under-positive-domination}

That positive operators inherit properties by means of domination is quite a common and rather classical theme in the theory of Banach lattices and is used extensively in many applications to partial differential equations.

We start this section by recalling a theorem of Vesentini (see \cite{Vesentini1968}; see also \cite[Theorem~1']{Vesentini1970}) that does not seem to have received much attention within abstract spectral theory. The only book reference the author could find is \cite[Theorem~3.4.7]{Aupetit2012}.

\begin{proposition}[Vesentini] \label{proposition:vesentini}
    Let $\calA$ be a complex Banach algebra. Let $\Omega\subseteq\C$ be a connected open set and let $F \colon \Omega \to \calA$ be holomorphic. Then the function
    \begin{align*}
        R \colon \Omega \to [-\infty, \infty), \quad R(z) \coloneqq \log(\mathrm{r}(F(z)))
    \end{align*}
    is subharmonic, where we use the convention $\log 0 = -\infty$.
\end{proposition}

The following proposition is an immediate consequence of Vesentini's theorem (see also \cite[Theorem~3.4.11]{Aupetit2012} for a slightly different formulation of the proof).

\begin{proposition} \label{proposition:peripheral-rigidity}
    Let $\calA$ be a complex unital Banach algebra. Let $\Omega\subseteq\C$ be a connected open set and let $F \colon \Omega \to \calA$ be holomorphic. Suppose that $\sigma(F(z)) \subseteq \overline{\bbD}$ for all $z \in \Omega$. Then $\sigma(F(z))\cap \bbT$ is independent of the choice of $z \in \Omega$.
\end{proposition}

\begin{proof}
    Fix $z_0 \in \Omega$ and $\mu \in \sigma(F(z_0)) \cap \bbT$ and consider the function
    \begin{align*}
        G \colon \Omega \to \calA, \qquad G(z) \coloneqq \frac{1 + \overline{\mu} F(z)}{2}.
    \end{align*}
    Then the spectral mapping theorem (see \cite[Theorem~10.28(b)]{Rudin1991}) yields
    \begin{align*}
        \sigma(G(z)) = \set[\bigg]{\frac{1+\overline{\mu}\lambda}{2} : \lambda \in \sigma(F(z))} \subseteq \overline{\bbD}.
    \end{align*}
    Thus, $\mathrm{r}(G(z)) \leq 1$ for all $z \in \Omega$ and $\mathrm{r}(G(z_0))=1$. As $G$ is holomorphic, Proposition~\ref{proposition:vesentini} implies that the function
    \begin{align*}
        R \colon \Omega \to [-\infty, \infty), \quad R(z) \coloneqq \log(\mathrm{r}(G(z)))
    \end{align*}
    is subharmonic. Moreover, $R(z) \leq 0$ for all $z \in \Omega$ and $R(z_0) = 0$. Therefore, it follows from the maximum principle for subharmonic functions (see \cite[Theorem~2.3.1(i)]{Ransford1995}) that $R(z) = 0$ and thus $\mathrm{r}(G(z)) = 1$ for all $z \in \Omega$. As $\sigma(G(z))$ is compact, there exists $\lambda \in \sigma(F(z))$ such that $\abs{1 + \overline \mu \lambda} = 2$. Since $\abs \lambda \leq 1$, the equality in the triangle inequality implies $\abs \lambda = 1$ and $\lambda = \mu$. Thus, $\mu \in \sigma(F(z)) \cap \bbT$ for all $z \in \Omega$, which shows the claim.
\end{proof}

We now bridge the gap from abstract spectral theory to a result on positive domination. Given the prior proposition, the proof is surprisingly straightforward, which should be taken as a sign of the strength of Vesentini's result.

\begin{theorem} \label{theorem:peripheral-domination}
    Let $S$ and $T$ be positive operators on a complex Banach lattice $E$ such that $0 \leq S \leq T$ and $\mathrm{r}(T) = 1$. Then one has
    \begin{align*}
        \sigma(S + \mu (T - S)) \cap \bbT = \sigma(S) \cap \bbT \subseteq \sigma(T) \cap \bbT
    \end{align*}
    for all $\mu \in \bbD$.
\end{theorem}

\begin{proof}
    Consider the holomorphic function
    \begin{align*}
        F \colon \bbD \to \calL(E), \quad F(\mu) \coloneqq S + \mu (T - S).
    \end{align*}
    Then one has
    \begin{align*}
        \abs{F(\mu) x} \leq (S + \abs \mu (T - S)) \abs x \leq T\abs x
    \end{align*}
    for all $\mu \in \bbD$ and $x \in E$. Iterating the inequality yields $\abs{F(\mu)^n x} \leq T^n \abs{x}$ for all $n \in \N$. Thus, $\norm{F(\mu)^n} \leq \norm{T^n}$ for all $n \in \N$, which implies $\mathrm{r}(F(\mu)) \leq \mathrm{r}(T) = 1$ by Gelfand's formula. Now it follows from Proposition~\ref{proposition:peripheral-rigidity} that
    \begin{align*}
        \sigma(S + \mu (T - S)) \cap \bbT = \sigma(S) \cap \bbT
    \end{align*}
    for all $\mu \in \bbD$. To prove the inclusion, fix $\lambda \in \sigma(S) \cap \bbT$. Then $\lambda I - F(t)$ is not invertible for any $0 \leq t < 1$. Since $F(t) \to T$ in operator norm as $t \to 1$ and the set of invertible operators is open, $\lambda I - T$ is not invertible either. Thus, $\lambda \in \sigma(T) \cap \bbT$.
\end{proof}

\begin{remark}
    Note that Theorem~\ref{theorem:peripheral-domination} uses positivity of both the operators $S$ and $T - S$. Mere modulus domination cannot replace this hypothesis: $S = -I$ and $T = I$ satisfy $\abs{Sx} \leq T \abs x$, but $-1 \in \sigma(S)$ and $-1 \notin \sigma(T)$.
\end{remark}

As a side effect of the theorem above, we immediately obtain a stronger version of a well-known theorem due to Räbiger and Wolff (see \cite[Theorem~1.4~and~Corollary~1.5]{Rabiger1997}).

\begin{corollary} \label{corollary:raebiger-wolff}
    Let $S$ and $T$ be positive operators on a complex Banach lattice $E$ such that $0 \leq S \leq T$ and $\mathrm{r}(S) = \mathrm{r}(T)$. Then one has the inclusion
    \begin{align*}
        \sigma(S) \cap \mathrm{r}(S)\bbT \subseteq \sigma(T) \cap \mathrm{r}(T) \bbT.
    \end{align*}
\end{corollary}

\begin{proof}
   If $\mathrm{r}(T) = 0$, the assertion is immediate. Otherwise, apply Theorem~\ref{theorem:peripheral-domination} to $S/\mathrm{r}(T)$ and $T/\mathrm{r}(T)$ and use the spectral mapping theorem.
\end{proof}

\begin{remark} \label{remark:rabiger-wolff}
    Räbiger and Wolff showed the above corollary under the additional assumption that $T$ is \emph{$(G)$-solvable} (see \cite[Corollary~1.5]{Rabiger1997}), i.e., there exist finitely many closed $T$-invariant ideals $\set{0} = I_0 \subseteq I_1 \subseteq I_2 \subseteq \ldots \subseteq I_n = E$, $n \in \N$, such that the operators $T_k$ induced by $T$ on $I_k/I_{k - 1}$, $1 \leq k \leq n$, satisfy the \emph{growth condition $(G)$}, given by
    \begin{align*}
        \limsup_{\lambda \downarrow \mathrm{r}(T_k)} \norm{(\lambda - \mathrm{r}(T_k)) R(\lambda, T_k)} < \infty.
    \end{align*}
    Depending on the specific reference, the notion of $(G)$-solvability is also referred to as \emph{Abel solvability}, whereas the growth condition $(G)$ is referred to as \emph{Abel boundedness}.

    What is rather interesting in the context of the present paper is that the cyclicity conjecture was also known to be true for $(G)$-solvable operators (see \cite[Theorem~4.9]{Lotz1968}) by a classical theorem of Lotz. In Appendix~\ref{section:g-solvable-case}, we will show that Lotz' theorem is indeed a consequence of the theorem of Räbiger and Wolff.
\end{remark}

\section{Critical minorants and cyclicity} \label{section:critical-minorants-and-cyclicity}

In this section, our aim is to construct a situation where we can apply Theorem~\ref{theorem:peripheral-domination} to prove the cyclicity conjecture for positive operators. For that purpose, we make use of a torsion operator constructed in Appendix~\ref{section:torsion-operators} and the concept of asymptotic rotational self-similarity. Using an ultrapower argument, we can improve this asymptotic rotational self-similarity to actual rotational self-similarity, which allows us to use the domination result. To that end, we start in a simpler setting that shows how the proof would work if the existence of rotationally self-similar minorants were already given. Recall that an operator $U \in \calL(E)$ on a complex Banach lattice $E$ is called a \emph{torsion operator} if $\abs{U x} = \abs x$ for each $x \in E$.

Let $T$ be a positive operator on a complex Banach lattice $E$ with $\mathrm{r}(T) = 1$. We call a positive operator $S \leq T$ a \emph{critical minorant} of $T$ if $\mathrm{r}(S) = 1$. A critical minorant $S$ of $T$ is called \emph{rotationally self-similar} to $T$ with respect to some $\lambda \in \bbT$ if there exists a torsion operator $U$ on $E$ such that
\begin{align*}
    S U = \lambda U S.
\end{align*}

For the definition and basic facts about torsion operators, we refer the reader to Appendix~\ref{section:torsion-operators}. It is not hard to see that if $S$ is rotationally self-similar with respect to some $\lambda \in \bbT$, this already implies cyclicity along the orbit of $\lambda$ as a consequence of Theorem~\ref{theorem:peripheral-domination}. Indeed, one has the following proposition.

\begin{proposition} \label{proposition:rotational-self-similarity-implies-cyclicity}
    Let $T$ be a positive operator on a complex Banach lattice $E$ with $\mathrm{r}(T) = 1$. Suppose that $T$ admits a minorant that is rotationally self-similar to it with respect to some $\lambda \in \bbT$. Then one has $\lambda^k \in \sigma(T)$ for all $k \in \Z$.
\end{proposition}

\begin{proof}
    Let $S$ be a minorant of $T$ that is rotationally self-similar to $T$ with respect to $\lambda$. Then there exists a torsion operator $U$ on $E$ such that $SU = \lambda US$. By Proposition~\ref{proposition:torsion-properties}(i), $U$ is an invertible isometry, so one has $\lambda S = U^{-1} S U$ and therefore $\lambda \sigma(S) = \sigma(S)$. Furthermore, since $S$ is positive and $\mathrm{r}(S) = 1$, one has $1 \in \sigma(S)$ by \cite[Proposition~V.4.1]{Schaefer1974}. Thus, it follows that $\lambda^k \in \sigma(S)$ for all $k \in \Z$ and Theorem~\ref{theorem:peripheral-domination} implies that $\lambda^k \in \sigma(T)$.
\end{proof}

Unfortunately, the general situation is not as simple as the setup of the prior proposition. More precisely, we generally cannot expect to find a rotationally self-similar minorant, as the following example demonstrates.

\begin{example} \label{example:irrational-shift}
    Let $E \coloneqq C(\bbT)$ and fix $\lambda \in \bbT$ which is not a root of unity. Consider the irrational shift
    \begin{align*}
        T \colon E \to E, \quad T f(z) \coloneqq f(\lambda z).
    \end{align*}
    Then $T$ is a positive invertible isometry and thus $\sigma(T) \subseteq \bbT$. Moreover, the functions $z \mapsto z^k$ are eigenvectors for the eigenvalues $\lambda^k$, $k \in \Z$. Since these eigenvalues are dense in $\bbT$, one has $\sigma(T) = \bbT$. We show that $T$ does not admit a minorant that is rotationally self-similar to it with respect to $-1$. For this, we just need the following two insights:
    \begin{enumerate}[\upshape (a)]
        \item Each positive operator $S$ on $E$ with $S \leq T$ and $\mathrm{r}(S) = 1$ coincides with $T$.
        \item There is no torsion operator $U$ such that $T U = -U T$.
    \end{enumerate}
    (a): First observe that every positive operator $S$ on $E$ with $S \leq T$ is of the form $S = M_a T$ for some $a \in C(\bbT)$ with $0 \leq a \leq 1$, where $M_a f \coloneqq af$. Indeed, one has $0 \leq Sf(z) \leq Tf(z) = f(\lambda z)$ for each $f \geq 0$ and $z \in \bbT$. Thus, the positive functional $f \mapsto Sf(z)$ is dominated by evaluation at $\lambda z$ and is therefore a scalar multiple of that evaluation. Hence, for each $z \in \bbT$ there exists $a(z) \in [0, 1]$ such that $S f(z) = a(z)f(\lambda z)$. For the special case $f = \one$, one gets that $a = S \one$, so $a \in C(\bbT)$.

    We next show that $\mathrm{r}(S) = 1$ implies $a = \one$ and thus $S = T$. Arguing by contraposition, suppose that $a \neq \one$. Then there exist a non-empty open set $O \subseteq \bbT$ and a constant $c < 1$ such that $a \leq c$ on $O$. Since the rotation $z \mapsto \lambda z$ is minimal, one has $N \in \N$ such that $\bbT = \bigcup_{n = 0}^{N - 1} \lambda^{-n} O$ (cf.~\cite[Proposition~3.3]{Eisner2015}). Thus, for every $z \in \bbT$, at least one of the elements $z, \lambda z, \ldots, \lambda^{N-1}z$ is contained in $O$. Consequently,
    \begin{align*}
    S^N \one(z)= \prod_{k = 0}^{N - 1} a(\lambda^k z) \leq c < 1
    \end{align*}
    for all $z \in \bbT$. Hence $\norm{S^N} = \norm{S^N \one}_\infty \leq c$ and therefore $\mathrm{r}(S) \leq c^{1/N} < 1$.

    (b): If $Tf=-f$, then $T^2f=f$. Since the rotation $z \mapsto \lambda^2 z$ is minimal, every continuous function fixed by $T^2$ is constant (see \cite[Proposition~3.6 \& Lemma~4.20]{Eisner2015}). Hence $f$ is constant, and $T f = -f$ implies $f = 0$, which shows that $-1 \notin \sigma_p(T)$. Now suppose that there were a torsion operator $U$ such that $T U = -U T$. Then $T U \one = -U T \one = - U \one$. On the other hand, $\abs{U \one} = \one$, so $U \one \neq 0$. Thus $-1 \in \sigma_p(T)$, which is clearly a contradiction.
\end{example}

We now turn to an asymptotic version of rotational self-similarity. Let $T$ be a positive operator on a complex Banach lattice $E$ with $\mathrm{r}(T)=1$. A sequence $(S_n)_{n \in \N}$ of critical minorants of $T$ is called
\emph{asymptotically rotationally self-similar} to $T$ with respect to some $\lambda \in \bbT$ if there exists a torsion operator $U$ on $E$ such that
\begin{align*}
    \norm{S_n U - \lambda U S_n} \to 0
\end{align*}
as $n \to \infty$.

\begin{remark}
    To motivate the following construction, note that one can characterise when an operator admits a rotationally self-similar minorant in the following way: Let $U$ be a torsion operator, $\lambda \in \bbT$, and consider the operator
    \begin{align*}
        \Psi \colon \calL(E) \to \calL(E), \quad \Psi(Q) \coloneqq \overline \lambda U^{-1} Q U.
    \end{align*}
    Then $S$ is a rotationally self-similar minorant of some positive operator $T$ with respect to $\lambda$ if and only if $\Psi(S) = S$, i.e., $S$ is a fixed point of $\Psi$. So if such an $S$ does not exist (as in Example~\ref{example:irrational-shift}) and one looks for a sequence $(S_n)_{n \in \N}$ of critical minorants of $T$ with respect to $\lambda$, one might be tempted to iterate $\Psi$ on $T$ to find such a sequence. However, $\Psi(T)$ does not even need to be real in general. On the other hand, the operator $\Re(\Psi(T)) = \frac 1 2 (\Psi(T) + \overline{\Psi(T)})$ is real, but in general not positive. Using that $U$ is a torsion operator, one can show that
    \begin{align*}
        - T \leq \frac 1 2 (\Psi(T) + \overline{\Psi(T)}) \leq T.
    \end{align*}
    So if we take another mean, this time with $T$, then we get that the operator
    \begin{align*}
        \Phi \colon \calL(E) \to \calL(E), \quad \Phi(Q) \coloneqq \frac{\Re(\Psi(Q)) + Q}{2} = \frac{2Q + \overline\lambda U^{-1} Q U + \lambda U Q U^{-1}}{4}
    \end{align*}
    is indeed positive. Thus, we can iterate $\Phi$ on $T$ without giving up the positivity of its iterates.
\end{remark}

\begin{proposition} \label{proposition:asymptotically-rotationally-self-similar-minorants}
Let $T$ be a positive operator on a Dedekind complete complex Banach lattice $E$ such that $\mathrm{r}(T)=1$. If $\lambda \in \sigma_p(T) \cap \bbT$, then $T$ admits a sequence of minorants that is asymptotically rotationally self-similar to it with respect to $\lambda$.
\end{proposition}

\begin{proof}
Let $0\neq z\in E$ satisfy $Tz=\lambda z$ and set $u\coloneqq\abs z$. By Proposition~\ref{proposition:torsion-operator-construction}, there exists a torsion operator $U$ on $E$ such that $Uu=z$.

Consider the operator
\begin{align*}
    \Phi \colon \calL(E) \to \calL(E), \quad \Phi(Q) \coloneqq \frac{2Q + \overline\lambda U^{-1} Q U + \lambda U Q U^{-1}}{4}
\end{align*}
and define the sequence $(S_n)_{n \in \N}$ recursively via $S_1 = \Phi(T)$ and $S_{n + 1} = \Phi(S_n)$ for all $n \in \N$. Then Proposition~\ref{proposition:torsion-properties}(i) yields
\begin{align*}
    \Phi(Q) x
    &= \frac{2Qx + \overline\lambda U^{-1} Q Ux + \lambda U Q U^{-1}x}{4} \\
    &= \frac{2Qx + \overline\lambda U^{-1} Q Ux + \overline{\overline\lambda U^{-1} Q Ux}}{4} \\
    &= \frac{Qx + \Re(\overline\lambda U^{-1} Q Ux)}{2}
\end{align*}
for each positive operator $Q$ on $E$ and $x \in E_+$. Since $\abs{\overline \lambda U^{-1} Q U x} \leq Q \abs{Ux} = Qx$, it follows that $0 \leq \Phi(Q) \leq Q$. Thus, it follows inductively that $0 \leq S_{n+1} \leq S_n \leq T$ for all $n \in \N$.

Next, we show that each $S_n$ is a critical minorant of $T$. Since $Tu \geq \abs{Tz} = u$, $Uu=z$, and $Tz=\lambda z$, one has
\begin{align*}
    S_1u = \Phi(T)u = \frac{Tu+\Re(\overline\lambda U^{-1} T U u)}{2} = \frac{Tu + u}{2} \geq u.
\end{align*}
Thus $S_1^m u\geq u$ for every $m\in\N$, so $\mathrm{r}(S_1)\geq1$. On the other hand, $0\leq S_1\leq T$ implies $\mathrm{r}(S_1)\leq \mathrm{r}(T)=1$, and hence $\mathrm{r}(S_1)=1$.

Now let $n\geq2$. Since $\Phi$ is positive and $\Phi(Q)\leq Q$ for each $Q\geq0$, one obtains
\begin{align*}
    T - S_n = \sum_{j = 0}^{n - 1} \Phi^j (T - S_1) \leq n (T - S_1).
\end{align*}
Consequently, $S_1 \leq (1-\frac 1 n) T + \frac 1 n S_n \leq T$. Since $\mathrm{r}(S_1)=\mathrm{r}(T)=1$, monotonicity of the spectral radius implies that $\mathrm{r}((1 - \frac 1 n) T + \frac 1 n S_n) = 1$, and Theorem~\ref{theorem:peripheral-domination} yields $\mathrm{r}(S_n)=1$ for all $n \in \N$.

It remains to prove that the sequence $(S_n)_{n \in \N}$ is asymptotically rotationally self-similar. Set
\begin{align*}
    p_{n,k}
    \coloneqq
    4^{-n}\binom{2n}{n+k}
    \qquad (-n\leq k\leq n),
\end{align*}
and set $p_{n,k}\coloneqq0$ otherwise. Consider the invertible isometry $\calA \colon \calL(E) \to \calL(E)$ given by $\calA(Q) \coloneqq \overline\lambda U^{-1} Q U$. Since $\Phi = (2I + \calA + \calA^{-1})/4 = \calA^{-1}(I + \calA)^2/4$, the binomial formula yields
\begin{align*}
    S_n
    =
    \sum_{k\in\Z}
    p_{n,k}\,
    \overline\lambda^{\,k}U^{-k}TU^k.
\end{align*}
Conjugation by $U$, followed by multiplication by $\overline\lambda$, shifts the coefficients by one. Since $U$ is an isometry and $(p_{n,k})_{k\in\Z}$ is symmetric and unimodal with maximum $p_{n,0}$, it follows that
\begin{align*}
    \norm{S_nU-\lambda US_n}
    &=
    \norm{\overline\lambda U^{-1} S_n U - S_n}\\
    &\leq
    \norm T\sum_{k\in\Z}\abs{p_{n,k-1}-p_{n,k}} \leq 2p_{n,0}\norm T =
    \frac{2}{4^n}\binom{2n}{n}\norm T
    \to 0
\end{align*}
as $n \to \infty$. Hence $(S_n)_{n \in \N}$ is asymptotically rotationally self-similar to $T$ with respect to $\lambda$.
\end{proof}

Let us briefly recall the concept of an
ultrapower of a Banach space $E$. Let $\calU$ be a free ultrafilter on $\N$ and consider the space $\ell^\infty(E)$ of bounded $E$-valued sequences endowed with its canonical norm $\norm{z}_\infty \coloneqq \sup_{n \in \N} \norm{z_n}$. Define the space
\begin{align*}
    c_{0, \calU}(E) \coloneqq \set{z \in \ell^\infty(E) : \lim_{n \to \calU} \norm{z_n} = 0}.
\end{align*}
The quotient space $E^\calU \coloneqq \ell^\infty(E)/c_{0, \calU}(E)$ is called the \emph{ultrapower} of $E$ with respect to the ultrafilter $\calU$. We write $x^\calU = (x_n)_\calU$ for the equivalence class of $(x_n)_{n \in \N}$ and recall that
\begin{align*}
    \norm{x^\calU} = \lim_{n \to \calU} \norm{x_n}.
\end{align*}
If $E$ is a complex Banach lattice, then $c_{0, \calU}(E)$ is a closed ideal and $E^\calU$ is again a complex Banach lattice with $\abs{x^\calU} = (\abs{x_n})_\calU$. We identify $E$ with its canonical isometric copy in $E^\calU$ via $x \mapsto (x)_\calU$, where $(x)_\calU$ denotes the equivalence class of the constant sequence $(x)_{n\in\N}$. Every bounded sequence $(T_n)_{n\in\N}$ in $\calL(E)$ induces an operator
\begin{align*}
    (T_n)_\calU \colon E^\calU \to E^\calU, \qquad (T_n)_\calU (x_n)_\calU \coloneqq (T_nx_n)_\calU,
\end{align*}
and one has $\norm{(T_n)_\calU} = \lim_{n \to \calU}\norm{T_n}$. In particular, every $T \in \calL(E)$ induces the ultrapower operator $T^\calU \coloneqq(T)_\calU$, and it is well known that $\sigma(T^\calU)=\sigma(T)$. For more information on ultrapowers, we refer to \cite[Section~V.1]{Schaefer1974}.

The following proposition shows that an ultrapower improves asymptotic rotational self-similarity to rotational self-similarity.

\begin{proposition} \label{proposition:rotational-self-similarity-in-ultrapowers}
    Let $T$ be a positive operator on a complex Banach lattice $E$ with $\mathrm{r}(T) = 1$. Moreover, let $\calU$ be a free ultrafilter on $\N$. Suppose that $T$ admits a sequence of minorants that is asymptotically rotationally self-similar to it with respect to some $\lambda \in \bbT$. Then the operator $T^\calU$ on $E^\calU$ admits a minorant that is rotationally self-similar to it with respect to $\lambda$.
\end{proposition}

\begin{proof}
    Let $(S_n)_{n \in \N}$ be a sequence of minorants that is asymptotically rotationally self-similar to $T$ with respect to $\lambda$ and let $U$ be the corresponding torsion operator. Then $\norm{S_n U - \lambda U S_n} \to 0$ as $n \to \infty$. Since $0 \leq S_n \leq T$, one has $\norm{S_n} \leq \norm T$ for each $n \in \N$, so the operator
    \begin{align*}
        S \colon E^\calU \to E^\calU, \quad S x^\calU \coloneqq (S_n x_n)_\calU
    \end{align*}
    is well-defined. Moreover, one clearly has $0 \leq S \leq T^\calU$. Note further that $U^\calU$ is a torsion operator. The sequence $(S_n U - \lambda U S_n)_{n \in \N}$ is mapped to the zero operator by the ultrapower embedding, so $S U^\calU = \lambda U^\calU S$. Finally, for every fixed $k \in \N$, one has
    \begin{align*}
        1 \leq \lim_{n \to \calU} \norm{S_n^k}^{1 / k} = \norm{S^k}^{1 / k} \leq \norm{T^k}^{1 / k}.
    \end{align*}
    Letting $k \to \infty$, we obtain $\mathrm{r}(S) = 1$ by Gelfand's formula. Thus, $S$ is a minorant that is rotationally self-similar to $T^\calU$ with respect to $\lambda$.
\end{proof}

\begin{corollary} \label{corollary:asymptotic-rotational-self-similarity-implies-cyclicity}
    Let $T$ be a positive operator on a complex Banach lattice $E$ with $\mathrm{r}(T) = 1$. If $T$ admits a sequence of minorants that is asymptotically rotationally self-similar to it with respect to some $\lambda \in \bbT$, then $\lambda^k \in \sigma(T)$ for all $k \in \Z$.
\end{corollary}

\begin{proof}
    Fix a free ultrafilter $\calU$ on $\N$. By Proposition~\ref{proposition:rotational-self-similarity-in-ultrapowers}, $T^\calU$ admits a minorant that is rotationally self-similar to it with respect to $\lambda$. Since $\mathrm{r}(T^\calU) = \mathrm{r}(T) = 1$, Proposition~\ref{proposition:rotational-self-similarity-implies-cyclicity} yields $\lambda^k \in \sigma(T^\calU) = \sigma(T)$ for every $k \in \Z$.
\end{proof}

We are now in a position to prove the main result. To that end, we use another classical ultrapower argument and synthesise the observations made in this section.

\begin{proof}[Proof of Theorem~\ref{theorem:cyclicity}]
    Fix $\lambda \in \sigma_{\mathrm{per}}(T)$. Fix a free ultrafilter $\calU$ on $\N$. We now pass to the space $\widehat E \coloneqq (E^\calU)''$ and the operator $\widehat T \coloneqq (T^\calU)''$. Then $\widehat E$ is Dedekind complete and $\widehat T$ is positive, since dual Banach lattices are Dedekind complete and adjoints of positive operators are positive (see \cite[Section~II.5]{Schaefer1974}). Moreover, spectra are preserved under passage to the bidual, so $\sigma(\widehat T) = \sigma(T^\calU) = \sigma(T)$ and $\mathrm{r}(\widehat T) = 1$. Since $\lambda \in \partial\sigma(T) \subseteq \sigma_{\mathrm{app}}(T)$, one has $\lambda \in \sigma_p(T^\calU)$. The canonical embedding $J \colon E^\calU \to (E^\calU)''$ satisfies $\widehat T J = J T^\calU$, so $\lambda \in \sigma_p(\widehat T) \cap \bbT$. Thus, Proposition~\ref{proposition:asymptotically-rotationally-self-similar-minorants} implies that $\widehat T$ admits a sequence of minorants that is asymptotically rotationally self-similar to it with respect to $\lambda$. Corollary~\ref{corollary:asymptotic-rotational-self-similarity-implies-cyclicity} now yields $\lambda^k \in \sigma(\widehat T) = \sigma(T)$ for every $k \in \Z$, which proves the cyclicity conjecture in full generality.
\end{proof}

\appendix

\section{Torsion operators} \label{section:torsion-operators}

In this appendix, we briefly recall some facts about the signum operator on complex Banach lattices and construct a torsion operator which is needed in the article. For a bounded operator $T$ on a complex Banach lattice, we write $\overline T x \coloneqq \overline{T \overline x}$. We recall the following consequence of the signum construction in \cite[Section~C-I.8]{NagelEd}:

\begin{proposition} \label{proposition:signum-operator}
    Let $E$ be a complex Dedekind complete Banach lattice and let $x \in E$. Then there exists a unique complex-linear operator $S_x \colon E \to E$ satisfying the following properties:
    \begin{enumerate}[\upshape (i)]
        \item $S_x \overline x = \abs x$, where $\overline x = \operatorname{Re} x - \ui \operatorname{Im} x$.
        \item $\abs{S_x y} \leq \abs y$ for each $y \in E$.
        \item $S_x y = 0$ for each $y \in E$ such that $x \perp y$.
    \end{enumerate}
\end{proposition}

\begin{proof}
    If $x = 0$, set $S_x = 0$. Otherwise, put $u \coloneqq \abs x$ and let $E_u \coloneqq \set{y \in E : \abs y \leq c u \text{ for some } c \geq 0}$ be the principal ideal generated by $u$. By \cite[Section~C-I.8]{NagelEd}, there exists a unique operator on $E_u$ satisfying (i) and (ii). This operator is order continuous and extends to the band $\set{u}^{\perp\perp}$ with the same domination by the order-continuous extension theorem for central operators (see the proof of \cite[Proposition~3.12(2)]{deJeu2025}). Since $E$ is Dedekind complete, this band is a projection band. Extending the operator by zero on $\set{u}^{\perp}$ yields the required operator $S_x$ on $E$. Finally, (ii) implies order continuity, so uniqueness follows from uniqueness on $E_u$, its order density in $\set{u}^{\perp\perp}$, and (iii).
\end{proof}

In the situation above, the operator $S_x$ is called the \emph{signum operator} with respect to $x$. For the next lemma, recall that in Dedekind complete Banach lattices every band is a projection band (see \cite[Theorem~1.2.9]{MeyerNieberg1991}).

\begin{lemma} \label{lemma:signum-properties}
    Let $E$ be a complex Dedekind complete Banach lattice and let $z \in E$. Let $P$ denote the band projection onto $\{\abs z\}^{\perp\perp}$. Then the following assertions hold:
    \begin{enumerate}[\upshape (i)]
        \item $S_z = S_zP$.
        \item $\abs{S_zx} = \abs x$ for each $x \in \{\abs z\}^{\perp\perp}$.
        \item $S_z \abs z = z$.
    \end{enumerate}
\end{lemma}

\begin{proof}
(i): By Proposition~\ref{proposition:signum-operator}(iii), $S_z$ vanishes on the band $\{\abs z\}^{\perp}$. Therefore, one has $S_z = S_z P + S_z(I - P) = S_z P$.

(ii): By (i), one has $\abs{S_z x} = \abs{S_z Px} \leq \abs{Px} = P \abs x$ for each $x \in E$. Thus, $S_z$ is regular and, since $E$ is Dedekind complete, its operator modulus exists and satisfies $0 \leq \abs{S_z} \leq P$. In particular, $S_z$ belongs to the ideal centre of $E$.
Moreover,
\begin{align*}
    \abs z = \abs{S_z\overline z} \leq \abs{S_z}\abs z \leq P\abs z = \abs z.
\end{align*}
Hence $\abs{S_z}\abs z=\abs z$. Both $\abs{S_z}$ and $P$ satisfy the three properties in Proposition~\ref{proposition:signum-operator} for the vector $\abs z$, so uniqueness yields $\abs{S_z} = P$. For operators in the ideal centre, the identity
\begin{align*}
    S_z\overline{S_z} = \overline{S_z}S_z = \abs{S_z}^2 = P
\end{align*}
holds (see~\cite[Theorem~3.1~and~Remark~3.2(1)]{deJeu2025}). Moreover, $\abs{\overline{S_z}y} = \abs{S_z\overline y} \leq \abs y$ for every $y \in E$. So, for $x \in \set{\abs z}^{\perp\perp}$, one has
\begin{align*}
\abs x = \abs{Px} = \abs{\overline{S_z} S_z x} \leq \abs{S_z x} \leq \abs x
\end{align*}
and therefore $\abs{S_z x} = \abs x$.

(iii): Conjugating $S_z\overline z = \abs z$ yields $\overline{S_z} z = \abs z$. Thus, the identity established above implies $S_z \abs z = S_z\overline{S_z}z = Pz = z$.
\end{proof}

Recall that an operator $U \in \calL(E)$ on a complex Banach lattice $E$ is called a \emph{torsion operator} if $\abs{U x} = \abs x$ for each $x \in E$.

\begin{proposition} \label{proposition:torsion-properties}
    Let $U$ be a torsion operator on a complex Banach lattice $E$. Then the following assertions hold:
    \begin{enumerate}[\upshape (i)]
        \item $U$ is an invertible isometry and $U^{-1} x = \overline{U \overline x}$ for all $x \in E$.
        \item $U$ is disjointness preserving, i.e., $U x \perp U y$ for all $x,y \in E$ with $x \perp y$.
        \item $U$ commutes with every band projection.
    \end{enumerate}
\end{proposition}

\begin{proof}
(i): Clearly, $U$ is an isometry. For $0 < u \in E$, the principal ideal $E_u$ is invariant under $U$ and $\overline U$. With its gauge norm and order unit $u$, this ideal can be identified with a space $C(K)$ (see \cite[Section~II.7]{Schaefer1974}). Under this identification, $U$ is multiplication by the function $a \coloneqq U\one$: indeed, for $f \in C(K)$ and $t \in K$, the modulus identity applied to $f - f(t)\one$ gives $(Uf)(t) = a(t)f(t)$. Moreover, $\abs a = \one$, so $\overline U$ is multiplication by $\overline a$ and $U\overline U = \overline U U = I$ on $E_u$. Since every vector belongs to a principal ideal, these identities hold on $E$. Thus, $U^{-1} = \overline U$, and its inverse is again a torsion operator.

(ii): This follows immediately from $\abs{Ux} = \abs x$ and $\abs{Uy} = \abs y$.

(iii): The modulus identity implies that $U$ leaves every band invariant. If $P$ is a band projection, both $PE$ and $(I - P)E$ are therefore invariant under $U$, so $PU = UP$.
\end{proof}

\begin{proposition} \label{proposition:torsion-operator-construction}
    Let $E$ be a complex Dedekind complete Banach lattice. For $z \in E$, there exists a torsion operator $U$ such that $U \abs z = z$.
\end{proposition}

\begin{proof}
Denote by $S_z$ the signum operator associated with $z$ and by $P$ the band projection onto $\set{\abs z}^{\perp\perp}$. We consider the operator
\begin{align*}
    U \colon E \to E, \qquad U \coloneqq S_z + I-P,
\end{align*}
and show that it has the desired properties.

First, let $x\in E$. By Lemma~\ref{lemma:signum-properties}, one has $S_z = S_zP$ and $\abs{S_z P x} = \abs{Px}$, so $S_z P x\in \set{\abs z}^{\perp\perp}$. Thus, the vectors $S_z P x$ and $(I - P)x$ are disjoint. Hence,
\begin{align*}
    \abs{Ux} = \abs{S_zPx+(I-P)x} = \abs{S_zPx}+\abs{(I-P)x} = \abs{Px}+\abs{(I-P)x} = \abs x.
\end{align*}
Thus, $U$ is a torsion operator. Moreover, since $P \abs z = \abs z$, it follows from Lemma~\ref{lemma:signum-properties} that $U \abs z = S_z \abs z = z$.
\end{proof}

\section{The \texorpdfstring{$(G)$}{(G)}-solvable case} \label{section:g-solvable-case}

We have seen in the proof of the main theorem that the domination result from Section~\ref{section:peripheral-spectral-inheritance-under-positive-domination} together with the approximation scheme carried out in Section~\ref{section:critical-minorants-and-cyclicity} implies the cyclicity conjecture. This raises one rather natural question. What happens if, instead of our domination result, one employs the weaker one of Räbiger and Wolff (cf.~Remark~\ref{remark:rabiger-wolff})? Before we answer this question, we prove the following technical lemma.

\begin{lemma} \label{lemma:criticality-under-growth}
    Let $T$ be a positive $(G)$-solvable operator on a complex Banach lattice $E$ such that $\mathrm{r}(T)=1$, and let $0 \leq S \leq T$. Suppose that there exist $0 \neq u \in E_+$ and $c \geq 0$ such that
    $Tu \geq u$ and
    \begin{align*}
        (T - S)u \leq c(Tu - u).
    \end{align*}
    Then one has $\mathrm{r}(S)=1$.
\end{lemma}

\begin{proof}
    First suppose that $T$ satisfies $(G)$. Following the proof of \cite[Theorem~1.4]{Rabiger1997}, consider the continuous lattice seminorm
    \begin{align*}
        p(x) \coloneqq \limsup_{t \downarrow 1}(t - 1)\norm{R(t, T) \abs x}
    \end{align*}
    and the closed ideal $J \coloneqq \ker p$. Since $p(Tx) \leq \norm T p(x)$, the ideal $J$ is $T$-invariant. Moreover, $0 \leq S \leq T$ implies that $J$ is also $S$-invariant.

    Since $Tu \geq u$, one has $(t-1)R(t,T)u \geq u$ for every $t>1$, and hence $p(u) \geq \norm u > 0$. On the other hand, the resolvent identity and $(G)$ yield
    \begin{align*}
        (t - 1)R(t, T)(Tu - u) &= (t - 1)^2 R(t, T)u - (t - 1) u \to 0
    \end{align*}
    as $t \downarrow 1$. Thus $Tu - u \in J$. The assumption $0 \leq (T - S)u \leq c(Tu - u)$ now implies that $(T - S) u \in J$, so $Su - u = Tu - u - (T - S)u \in J$.

    Consequently, the operator $\widetilde S$ induced by $S$ on $E/J$ satisfies $\widetilde S(u + J) = u + J \neq 0$. Therefore,
    \begin{align*}
        1 \leq \mathrm{r}(\widetilde S) \leq \mathrm{r}(S) \leq \mathrm{r}(T) = 1.
    \end{align*}

    Since $T$ is $(G)$-solvable, there exist ideals $\set{0} = I_0 \subseteq I_1 \subseteq \cdots \subseteq I_m = E$. Denote by $T_j$ and $S_j$ the operators induced by $T$ and $S$ on $I_j/I_{j-1}$.

    Choose the least $j \in \set{0, \ldots, m}$ such that $u \in I_j$ and set
    $u_j \coloneqq u + I_{j - 1}$. Then $u_j \neq 0$,
    $T_j u_j \geq u_j$, and
    \begin{align*}
        0 \leq (T_j - S_j) u_j
        \leq c(T_j u_j - u_j).
    \end{align*}
    Moreover, $T_j u_j \geq u_j \neq 0$ implies that $\mathrm{r}(T_j) \geq 1$, while $\mathrm{r}(T_j) \leq \mathrm{r}(T) = 1$. Since $T_j$ satisfies $(G)$, the first part of the proof yields $\mathrm{r}(S_j)=1$. Hence, one obtains $1 = \mathrm{r}(S_j) \leq \mathrm{r}(S) \leq 1$.
\end{proof}

We can recover the following classical theorem of Lotz (see \cite[Theorem~4.9]{Lotz1968}) using the theorem of Räbiger and Wolff.

\begin{theorem}[Lotz] \label{theorem:lotz}
    Let $T$ be a positive $(G)$-solvable operator on a complex Banach
    lattice $E$ such that $\mathrm{r}(T)=1$. Then
    $\lambda^k \in \sigma(T)$ for all
    $\lambda \in \sigma_{\mathrm{per}}(T)$ and $k \in \Z$.
\end{theorem}

\begin{proof}[Proof sketch]
    Repeat the proof of Theorem~\ref{theorem:cyclicity}, observing that $(G)$-solvability is preserved under ultrapowers and passage to the bidual. The only changes concern the two applications of Theorem~\ref{theorem:peripheral-domination}: in the construction of the asymptotically rotationally self-similar minorants, Lemma~\ref{lemma:criticality-under-growth} yields $\mathrm{r}(S_n)=1$, while the final spectral inclusion follows from \cite[Corollary~1.5]{Rabiger1997}. All remaining steps can be carried out unchanged.
\end{proof}

\subsection*{Acknowledgements} The author thanks Rainer Nagel and Noa Bihlmaier for detailed feedback on an early draft version of the paper and Jochen Glück for reading an earlier draft version of the article.

\bibliographystyle{plain}
\bibliography{literature}

\end{document}